\documentclass[12pt]{amsart}
\usepackage{graphicx,amsmath,amsfonts,amsthm,amssymb,verbatim,stmaryrd,fullpage}
\newtheorem{theorem}{Theorem}
\newtheorem{lemma}[theorem]{Lemma}
\newtheorem{prop}[theorem]{Proposition}

\newtheorem{cor}[theorem]{Corollary}

\theoremstyle{remark}

\newcommand{\R}{\mathbb R}
\newcommand{\C}{\mathbb C}
\newcommand{\Z}{\mathbb Z}
\newcommand{\Q}{\mathbb Q}
\newcommand{\F}{\mathbb F}
\newcommand{\A}{\mathbb A}

\newcommand{\p}{\mathfrak p}

\newcommand{\VV}{\mathcal V}

\newcommand{\GG}{\mathcal G}
\newcommand{\Hh}{\mathcal H}
\newcommand{\PP}{\mathcal P}
\newcommand{\TT}{\mathcal T}

\newcommand{\Sym}{\text{Sym}}

\newcommand{ \Galg}{ \F_p \llbracket G \rrbracket }

\begin{document}

\title{Bounds for Multiplicities of Automorphic Forms of Cohomological Type on $GL_2$}

\author{Simon Marshall}
\address{The Institute for Advanced Study\\
Einstein Drive\\
Princeton\\
NJ 08540, USA}
\email{slm@math.princeton.edu}

\begin{abstract}
We prove a power saving for the dimension of the space of cohomological automorphic forms of fixed level and growing weight on $GL_2$ over any number field which is not totally real.  Our proof involves the theory of $p$-adically completed cohomology developed by Calegari and Emerton, and a bound for the growth of coinvariants in certain finitely generated non-commutative Iwasawa modules.
\end{abstract}

\maketitle

\section{Introduction}

Let $F$ be a number field of degree $n$ which is not totally real, with $r_1$ real places and $r_2$ complex places, and with ring of Adeles $\A$ and finite Adeles $\A_f$.  Let $F_\infty = F \otimes_\Q \R$, so that $GL_2(F_\infty) = GL_2(\R)^{r_1} \times GL_2(\C)^{r_2}$, and let $Z_\infty$ be the centre of $GL_2(F_\infty)$.  Let $K_f$ be a compact open subgroup of the finite Adele group $GL_2(\A_f)$, and define $X = GL_2(F) \backslash GL_2(\A) / K_f Z_\infty$.  If $ {\bf d} = (d_1, \ldots, d_{r_1 + r_2})$ is an $(r_1 + r_2)$ - tuple of positive even integers, we shall let $S_{\bf d}(K_f)$ denote the space of cusp forms on $X$ which are of cohomological type with weight ${\bf d}$, and define $\Delta({\bf d})$ to be

\begin{equation*}
\Delta({\bf d}) = \prod_{i \le r_1 } d_i \times \prod_{i > r_1 } d_i^2.
\end{equation*}
In this paper, we shall investigate the dimension of $S_{\bf d}(K_f)$ as ${\bf d}$ varies with $K_f$ held fixed.  When $F$ is totally real, Shimizu \cite{Sh} has proven that

\begin{equation*}
\dim S_{\bf d}(K_f) \sim C \Delta( {\bf d} )
\end{equation*}
for some constant $C$ independent of ${\bf d}$, while if $F$ is not totally real it may be proven easily using the trace formula that

\begin{equation}
\label{littleo}
\dim S_{\bf d}(K_f) = o( \Delta( {\bf d} ) ).
\end{equation}
The puropse of this paper is to strengthen (\ref{littleo}) by a power in the case where some entries of ${\bf d}$ are held fixed while the rest grow uniformly.  To be precise, if $I$ is a subset of $[ 1, \ldots, r_1 + r_2]$ and ${\bf n} = ( n_i | i \in I )$ is an $|I|$-tuple of positive even integers, we define $\mathcal{D}({\bf n})$ to be the set of weights ${\bf d}$ such that $d_i = n_i$ for $i \in I$.  We then prove:

\begin{theorem}
\label{main}
Assume that $F$ is not totally real.  Then there is a $\delta > 0$ depending only on $F$ such that for any fixed $K_f$, $I$ and ${\bf n}$, we have

\begin{equation}
\label{weight}
\dim S_{\bf d}(K_f) \ll (\min_{i \, \notin I} d_i)^{-\delta} \Delta({\bf d})
\end{equation}
for all ${\bf d} \in \mathcal{D}({\bf n})$.

\end{theorem}

{\bf Remark}: If the prime $p$ splits completely in $F$, we may take $\delta = (\ln 3p^2 - \ln (2p^2 + 1)) / 2 \ln p$.  This has a minimum over $p$ of $\delta \simeq 0.207 \ldots$ when $p=2$.\\

Note that theorem \ref{main} strengthens (\ref{littleo}) by a power if we restrict to ${\bf d} \in \mathcal{D}({\bf n})$ such that 

\begin{equation*}
c \le \ln d_i / \ln d_j \le C, \quad i, j \notin I
\end{equation*}
for some $C, c > 0$.  It is interesting to compare our theorem with results of Finis, Grunewald and Tirao \cite{FGT} in the case when $F$ is imaginary quadratic.  They prove the bounds

\begin{equation}
\label{trace}
d \ll \dim S_{\bf d}(K_f) \ll \frac{d^2}{\ln d}, \quad {\bf d} = (d)
\end{equation}
using base change and the trace formula respectively, where the upper bound is valid for any $K_f$ and the lower bound for any $K_f$ contained in the product of the standard maximal compact subgroups of the $p$-adic groups $GL_2(F_\p)$.  For imaginary quadratic $F$ theorem \ref{main} reads

\begin{equation*}
\dim S_{\bf d}(K_f) \ll d^{2-\delta},
\end{equation*}
and (\ref{trace}) demonstrates that the actual growth rate of $\dim S_{\bf d}(K_f)$ is a smaller power of $d$ (which is probably $d$, as the experimental data of \cite{FGT} shows).  When $F$ is contained in a solvable extension of its maximal totally real subfield $F_0$ and ${\bf d} = (d, \ldots, d)$ is parallel, Rajan \cite{Ra} has also used base change to show that

\begin{equation*}
\dim S_{\bf d}(K_f) \gg d^{ | F_0 : \Q | }
\end{equation*}
after shrinking $K_f$ if necessary  (note the distinction between this result and that of \cite{FGT}, which shows that this lower bound holds for $K_f$ maximal).\\

Automorphic forms in $\dim S_{\bf d}(K_f)$ are tempered but not in the discrete series, and bounds for the multiplicities of such forms which improve over the trivial bound by a power are quite rare.  Indeed, the best known bounds for tempered multiplicities that may be proven using purely analytic methods such as the trace formula only strengthen the trivial bound by a power of log, and to obtain more than this it seems necessary to exploit some additional number theoretic or cohomological properties of the automorphic forms.  To our knowledge, there are only two other families of automorphic forms for which bounds of this kind are known.  The first of these is $S_1(q)$, the space of classical holomorphic forms of weight 1, level $q$ and character the Legendre symbol $( \tfrac{ \cdot }{q} )$, for which bounds were proven by Duke \cite{D} and Michel and Venkatesh \cite{MV} using the restrictions placed on the Fourier coefficients of such forms by the theorem of Deligne and Serre.

The second is the collection of automorphic forms of cohomological type appearing in a `$p$-adic congruence tower', studied by Calegari and Emerton in \cite{CE1}.  Here they prove a bound for the multiplicity of cohomological forms of fixed weight and full level $Np^k$ with $k \rightarrow \infty$ on any reductive group $G$, provided the form makes a contribution to cohomology outside of the degree in which the discrete series of $G$ (if any) appears.  One of the interesting features of the proof of theorem \ref{main} is that it draws heavily on the methods used by Calegari and Emerton, in spite of the differences between the families of automorphic forms the two results deal with.

{\bf Acknowledgements}:  We would like to thank Matthew Emerton and Peter Sarnak for many helpful discussions, and Frank Calegari for helping us to understand an earlier, conditional form of our theorem in more depth, which eventually lead to the current version.  We acknowledge the generous support of the Institute for Advanced Study while this work was being conducted.

\section{Notation and outline of proof}

Let us first give a rough outline of the proof of theorem \ref{main}, which we shall expand on in the remainder of the section.  Let $Y$ be the locally symmetric space attached to $X$.  Because the forms in $S_{\bf d}$ are of cohomological type, bounding their multiplicity is equivalent to bounding the cohomology of certain complex local systems $W_{\bf d}$ on $Y$, and because $Y$ was arithmetic we are able to replace $W_{\bf d}$ with analogous systems $V_{\bf d}$ over $\Q_p$.  By choosing a lattice in $V_{\bf d}$ and reducing mod $p$ it will suffice to bound the $\F_p$ homology of a family of congruence covers of $Y$.  The family of covers that appears is sufficiently similar to the kind studied by Calegari and Emerton that we may apply their theory of $p$-adically completed homology, which converts the statement about $\F_p$ growth that we require into one about the coinvariants of noncommutative Iwasawa modules.  This is proposition \ref{induced} below, whose proof will be discussed seperately in section \ref{coinvar} and which should be regarded as the key ingredient in theorem \ref{main}.

\subsection{The Eichler-Shimura isomorphism}

$X$ may be written as a disjoint union

\begin{equation*}
X = \coprod_{i=1}^N \Gamma_i \backslash SL_2(F_\infty),
\end{equation*}
where $\Gamma_i$ are lattices of the form $SL_2(F) \cap K_i$ for compact open subgroups $K_i$ of $SL_2(\A_f)$, and by shrinking $K_f$ if necessary we may assume that $\Gamma_i$ are torsion free.  We define

\begin{eqnarray*}
Y & = & \coprod_{i=1}^N \Gamma_i \backslash SL_2(F_\infty) / K_\infty \\
& = & \coprod_{i=1}^N Y_i
\end{eqnarray*}
to be the associated locally symmetric spaces, where $K_\infty \subset SL_2(F_\infty)$ is the standard maximal compact subgroup.  We define $W_{\bf d}$ to be the representation of $SL_2(F_\infty)$ obtained by taking the tensor product of the representation $\Sym^{d_i-2}$ of $SL_2(F_{v_i})$ when $v_i$ is a real place, and the representation $\Sym^{d_i/2 -1} \otimes \overline{\Sym}^{d_i/2-1}$ of $SL_2(F_{v_i})$ when $v_i$ is complex.  We also use $W_{\bf d}$ to denote the local system on $Y$ obtained by restricting $W_{\bf d}$ to each of the groups $\Gamma_i$.\\

Let $H^i(Y, W_{\bf d})$ be the cohomology groups of the local system $W_{\bf d}$, and $H^i_c(Y, W_{\bf d})$ the subspace of classes whose restriction to some neighbourhood of the cusps is trivial.  It follows from the Eichler-Shimura isomorphism (see \cite{Hr}, section 3 or \cite{B1}, theorem 3.5 and \cite{B2}, corollary 5.5) that if $\dim W_{\bf d} > 1$,

\begin{equation}
\label{ES}
\dim H_c^{r_1+r_2} ( Y, W_{\bf d} ) = 2^{r_1} \dim S_{\bf d}(K_f).
\end{equation}
Using the duality between $H^i_c$ and $H_i$, we see that theorem \ref{main} would be implied by the following proposition:

\begin{prop}
\label{main2}
Let $Y = SL_2(F) \backslash SL_2(\A) / K_f K_\infty$ for some compact open $K_f \subset SL_2(\A_f)$, and $I$ and ${\bf n}$ be as in theorem \ref{main}.  Then there is a $\delta > 0$ depending only on $F$ such that for any fixed $K_f$, $I$ and ${\bf n}$ we have

\begin{equation}
\label{weight2}
\dim H_i( Y, W_{\bf d} ) \ll (\min_{i \, \notin I} d_i)^{-\delta} \Delta({\bf d})
\end{equation}
for all $i$ and all ${\bf d} \in \mathcal{D}({\bf n})$.
\end{prop}

{\bf Remark}: While the notational convention we have introduced for $W_{\bf d}$ allows us to state the Eichler-Shimura isomorphism (\ref{ES}) simply, it will be more convenient in the remainder of the paper to adopt the following convention: let $\{ \sigma_1, \ldots, \sigma_n \}$ be the complex embeddings of $F$, and let ${\bf d}$ be an $n$-tuple of non-negative integers indexed by the $\sigma_i$ for which $d_i = d_j$ when $\sigma_i$ and $\sigma_j$ are complex conjugates.  We then let $W_{\bf d}$ be the representation of $SL_2(F_\infty)$ obtained by forming the tensor product of the representations $\Sym^{d_i}$ of $SL_2(F_{v_i})$ when $\sigma_i$ corresponds to a real place $v_i$, and $\Sym^{d_i} \otimes \overline{\Sym}^{d_i}$ of $SL_2(F_{v_i})$ when $\sigma_i$ is either of the two embeddings corresponding to a complex place $v_i$.\\

\subsection{Completed homology}

Define the compact $p$-adic analytic subgroups $G, G(p^k), H(p^k)$ and $T(p^k)$ of $SL_2(\Z_p)$ by

\begin{eqnarray*}
& G(p^k) = \left\{ \left( \begin{array}{cc} a & b \\ c & d \end{array} \right) \Big| a-1, b, c, d-1 \in p^k \Z_p \right\}, \quad G = G(p), & \\
& H(p^k) = \left\{ \left( \begin{array}{cc} a & b \\ c & d \end{array} \right) \Big| b \in p^k \Z_p \right\} \cap G, \quad \text{and} \quad T(p^k) = \left\{ \left( \begin{array}{cc} a & b \\ c & d \end{array} \right) \Big| b, c \in p^k \Z_p \right\} \cap G. &
\end{eqnarray*}
Moreover, if $t \ge 1$ is an integer and ${\bf k} = (k_1, \ldots, k_t)$ is a $t$-tuple of positive integers, define

\begin{equation}
\label{Gdef}
\GG = \prod_{i=1}^t G_i, \quad \GG_{\bf k} = \prod_{i=1}^t G_i(p^{k_i}), \quad \Hh_{\bf k} = \prod_{i=1}^t H_i(p^{k_i}) \quad \text{and} \quad \TT_{\bf k} = \prod_{i=1}^t T_i(p^{k_i}),
\end{equation}
where $G_i \simeq G$ for all $i$, etc.  In section \ref{compcoh} we shall prove that if we choose $t = n$, there exists an algebraic representation $V_{\bf d}$ of $\GG$ over $\Q_p$ and an injection $\Gamma \rightarrow \GG$ such that if we also let $V_{\bf d}$ denote the local system on $Y$ obtained by restricting $V_{\bf d}$ to $\Gamma$, we have

\begin{equation}
\label{changescalar}
\dim_\C H_i( Y, W_{\bf d} ) = \dim_{\Q_p} H_i( Y, V_{\bf d} ).
\end{equation}

The advantage of this $p$-adic reformulation is that we may study the RHS of (\ref{changescalar}) using the completed homology modules defined by Emerton in \cite{E} and studied by Calegari and Emerton in \cite{CE1, CE2}.  These are finitely generated modules over non-commutative Iwasawa algebras $\Lambda_{\Q_p}$ and $\Lambda$, defined to be

\begin{eqnarray*}
\Lambda_{\Q_p} & = & \Lambda_{\Z_p} \otimes_{\Z_p} \Q_p, \quad \Lambda_{\Z_p} = \underset{ \substack{ \longleftarrow \\ {\bf k} } }{\lim} \: \Z_p[ \GG / \GG_{\bf k} ], \\
\Lambda & = & \underset{ \substack{ \longleftarrow \\ {\bf k} } }{\lim} \: \F_p[ \GG / \GG_{\bf k} ],
\end{eqnarray*}
where the projections are given by the trace maps $\Z_p[ \GG / \GG_{\bf k'} ] \rightarrow \Z_p[ \GG / \GG_{\bf k} ]$ for ${\bf k}' \ge {\bf k}$.  We shall describe the structure of these algebras and their associated modules in more depth in section \ref{coinvar}.  In section \ref{compcoh} we shall define finitely generated $\Lambda_{\Q_p}$ modules $\widetilde{H}_j( \VV_{\bf d})$, which may be used to calculate the RHS of (\ref{changescalar}) via a spectral sequence

\begin{equation*}
E^{i,j}_2 = H_i ( \GG_\PP, \widetilde{H}_j(\VV_{\bf d}) \otimes_{\Z_p} \Q_p ) \Longrightarrow H_{i+j}( Y, V_{\bf d}).
\end{equation*}

This spectral sequence allows us to reduce our problem to one of bounding the dimension of $H_{i, \text{con}} ( \GG, M \otimes V_{\bf d})$, where $M$ is a fixed torsion $\Lambda_{\Q_p}$ module and ${\bf d}$ varies (note that hereafter we shall always assume homology groups of $\GG$ to be computed with continuous chains, and drop the subscript `con').  We shall do this by choosing lattices $L \subset M$ and $\VV_{\bf d} \subset V_{\bf d}$, so that $L \otimes \VV_d$ is a lattice in $M \otimes V_{\bf d}$, and applying the bound

\begin{equation}
\label{reduction}
\dim_{\Q_p} H_i ( \GG, M \otimes V_{\bf d}) \le \dim_{\F_p} H_i ( \GG, M' \otimes (\VV_{\bf d}/p) ), \quad M' = L/p.
\end{equation}
In section \ref{cosetrep} we shall prove that $\VV_{\bf d}$ may be chosen to have the property that $\VV_{\bf d}/p$ is a submodule of $\F_p[\GG / \Hh_{\bf k} ]$ for any ${\bf k}$ satisfying $p^{k_i-1} > d_i$ for all $i$.  We will also provide a characterisation of the submodules of $\F_p[\GG / \Hh_{\bf k} ]$ that allows us to apply Shapiro's lemma to the RHS of (\ref{reduction}), and transfers the problem to one of bounding $H_i( \Hh_{\bf k}, M')$ for $M'$ a fixed torsion $\Lambda$ module and ${\bf k}$ varying.  Finally, by using the commensurator of $\Gamma$ to replace $\Hh_{\bf k}$ with $\TT_{\bf k}$, it suffices to bound $H_i( \TT_{\bf k}, M')$.  The trivial bound for the dimension of this space is

\begin{equation}
\label{trivbd}
\dim H_i( \TT_{\bf k}, M') \ll | \GG : \TT_{\bf k} |,
\end{equation}
and it turns out that the reductions we have made are tight in the sense that we may recover the trivial bound for $S_{\bf d}(K_f)$ from (\ref{trivbd}).  It therefore suffices to make any power improvement in (\ref{trivbd}), and this is provided by the following proposition, which lies at the heart of our proof (we define the rank $r$ of a $\Lambda$ module in section \ref{coinvar}; it suffices here to know that $r = 0$ when $M$ is torsion).

\begin{prop}
\label{induced}
For any $t \ge 1$, let $\GG$ and $\TT_{\bf k}$ be as in (\ref{Gdef}) and let $M$ be a finitely generated $\Lambda$ module of rank $r$.  Then

\begin{eqnarray}
\label{h0}
\dim M_{\TT_k} & = & (r + O(\eta^\kappa)) | \GG : \TT_{\bf k} |, \\
\label{hi}
\dim H_i( \TT_k, M) & \ll & \eta^\kappa | \GG : \TT_{\bf k} |, \quad i \ge 1,
\end{eqnarray}
for all $t$ - tuples ${\bf k}$, where $\kappa = \min(k_i)$ and $\eta = \tfrac{2p^2 +1 }{p^2}$.
\end{prop}

{\bf Remark}: We have attempted to extend this proof to higher rank groups, but have so far been prevented from doing so by the fact that the subgroup $\Hh$ for which we can realise $\VV_{\bf d}$ as a subrepresentation of $\F_p[ \GG / \Hh]$, and the subgroups $\TT$ for which we are able to prove analogues of proposition \ref{induced}, are not conjugate to one another under the noncompact $p$-adic group containing them.

{\bf Structure of Paper}: Proposition \ref{induced} is proven in section \ref{coinvar}, and section \ref{cosetrep} contains results on the structure of $\F_p[\GG / \Hh_{\bf k}]$ including how to choose the lattice $\VV_{\bf d} \subset V_{\bf d}$ such that $\VV_{\bf d}/p \subset \F_p[\GG / \Hh_{\bf k}]$.  We apply the theory of $p$-adically completed cohomology in section \ref{compcoh}, and combine these ingredients in section \ref{modp} to conclude the proof.

\section{Coinvariants of $\Lambda$-modules}
\label{coinvar}

This section contains the proof of proposition \ref{induced}.  Once we have proven the case $i=0$ and $r=0$ the others will follow by the same arguments used by Harris \cite{H} and Calegari and Emerton \cite{CE1}, and so we restate this case as a seperate proposition:

\begin{prop}
\label{coinvarlem}
If $t \ge 1$, $\GG$ and $\TT_{\bf k}$ are as in (\ref{Gdef}) and $M$ is any torsion $\Lambda$ module, we have

\begin{equation}
\label{coinvarbd}
\dim M_{\TT_k} \ll \eta^\kappa | \GG : \TT_{\bf k} |,
\end{equation}
for all ${\bf k}$, where $\kappa = \min(k_i)$ and $\eta = \tfrac{2p^2 +1 }{p^2}$.
\end{prop}

We begin the proof by describing the structure of the algebras $\Lambda$ and $\Lambda_{\Q_p}$ in more detail (although we will not work with $\Lambda_{\Q_p}$ in this section, it is convenient to present its structure theory together with that of $\Lambda$).  $\Lambda$ is a non-commutative Noetherian integral domain, and so its field of fractions $\mathcal{L}$ is a division ring which is flat over $\Lambda$ on both sides (and likewise for $\Lambda_{\Q_p}$ and its field of fractions $\mathcal{L}_{\Q_p}$).  If $M$ is a finitely generated $\Lambda$ (respectively $\Lambda_{\Q_p}$) module, then $M \otimes_{\Lambda} \mathcal{L}$ (resp. $M \otimes_{\Lambda_{\Q_p}} \mathcal{L}_{\Q_p}$) is a finite dimensional $\mathcal{L}$ (resp. $\mathcal{L}_{\Q_p}$) vector space, and we define the rank of $M$ to be the dimension of this vector space.  We see that rank is additive in short exact sequences by the flatness of $\mathcal{L}$ over $\Lambda$, and that $M$ has rank 0 if and only if it is torsion.

The basic result on the growth of coinvariants in a finitely generated module $M$ for $\Lambda$ or $\Lambda_{\Q_p}$ is due to Harris \cite{H}.  To state it in the case under consideration, let $\GG_n$ denote the group $\GG_{\bf k}$ with ${\bf k}$ chosen to be $(n, \ldots, n)$.  We then have

\begin{theorem}
\label{harris}
Let $M$ be a finitely generated module for either $\Lambda_{\Q_p}$ or $\Lambda$ of rank $r$.  We then have

\begin{equation}
\label{harriseqn}
\dim M_{\GG_n} = r | \GG : \GG_n | + O( p^{(3t-1)n} ).
\end{equation}

\end{theorem}

The significance of the exponent $3t-1$ is that it is one less than the dimension of $\GG$, so that $| \GG : \GG_n | = c p^{3tn}$ for some $c$ and hence the error term in theorem \ref{harris} is smaller than the main term by a power.  When $M$ is torsion, (\ref{harriseqn}) becomes

\begin{equation*}
\dim M_{\GG_n} \ll | \GG : \GG_n |^{1 - 1/3t}
\end{equation*}
so that proosition \ref{coinvarlem} is an extension of theorem \ref{harris} for $\Lambda$-modules to the family of subgroups $\TT_{\bf k}$.  The key difference between these two cases is that the groups $\TT_{\bf k}$ are not shrinking uniformly to the identity.\\

To prove lemma \ref{coinvarlem}, we shall first establish a quantitative version of it in the case $t = 1$ (proposition \ref{single} below), which we then use to perform induction on $t$.  The reason a two part argument of this kind is necessary is that the cases $t = 1$ and $t > 1$ are qualitatively different.  When $t = 1$, theorem \ref{harris} immediately gives the bound $\dim M_{\GG_k} \ll p^{2k}$ by assumption that $M$ is torsion, so that we only need to show that the space of coinvariants shrinks by a factor of $p^{\delta k}$ for some $\delta > 0$ when we pass to the group $\TT_k$.  In contrast, when $t = 2$ and $k_1 = k_2 = k$ the best general bound for $M_{\GG_k}$ is $\dim M_{\GG_k} \ll p^{5k}$, so that we must prove a much more substantial saving of $p^{-(1+\delta)k}$ on passing to $\TT_k$.  We shall discuss this point in more detail after the proof of the base case.

\begin{prop}
\label{single}
Let $M$ be a $\Lambda = \Galg$ module, and suppose that for some $k > 0$, $M$ satisfies the bound

\begin{equation}
\label{growthhyp}
\dim M_{G(p^l)} \le C p^{2l}
\end{equation}
for all $l \le k$.  Then we have

\begin{equation*}
\dim M_{T(p^k)} \le C \eta^{k-2} p^{2k}.
\end{equation*}

\end{prop}

\begin{proof}

Because we only need to make a small saving, we are able to use an argument which essentially assumes that $T(p^k)$ fixes everything in $M_{G(p^k)}$ and derives a contradiction.  Indeed, if this is the case then all groups conjugate to $T(p^k)$ under the action of $G$ must also act trivially in $M_{G(p^k)}$, but the subgroup generated by these conjugates contains $G(p^2)$ and so we obtain the stronger bound $\dim M_{G(p^k)} \ll 1$.

In practice, we shall apply an inclusion-exclusion argument based on this idea to a family of subgroups which interpolates between $G(p^k)$ and $T(p^k)$.  For $0 \le j \le l-1$, let $T(l,j)$ be the group

\begin{equation*}
T(l,j) = \left\{ \left( \begin{array}{cc} a & \\ & a^{-1} \end{array} \right) | a \equiv 1 (p^{l-j}) \right\} + G(p^l),
\end{equation*}
so that for fixed $l$, $T(l,j)$ forms a family growing from $T(l,0) = G(p^l)$ to $T(l,l-1) = T(p^l)$.  We shall prove by induction on pairs $(l,j)$, ordered lexicographically, that

\begin{equation}
\label{indhyp}
\dim M_{T(l,j)} \le C \eta^{j-1} p^{2l}.
\end{equation}

To do this, we apply inclusion-exclusion counting to the coinvariants of $M$ under the subgroups $T(l,j)$, $T(l,j)'$ and $T(l,j)''$, where the second two groups are conjugates of the first defined by

\begin{equation*}
T(l,j)' = N_j T(l,j) N_j^{-1}, \qquad T(l,j)'' = \overline{N}_j T(l,j) \overline{N}_j^{-1},
\end{equation*}
and $N_j$ and $\overline{N}_j$ are 

\begin{equation*}
N_j = \left( \begin{array}{cc} 1 & p^{j-1} \\ & 1 \end{array} \right), \qquad \overline{N}_j = \left( \begin{array}{cc} 1 &  \\ p^{j-1} & 1 \end{array} \right).
\end{equation*}

It may be seen that $T(l,j)'$ and $T(l,j)''$ are generated by the elements

\begin{equation*}
\left( \begin{array}{cc} 1+p^{l-j} & u p^{l-1} \\ & (1+p^{l-j})^{-1} \end{array} \right), \qquad \left( \begin{array}{cc} 1+p^{l-j} & \\ u' p^{l-1} & (1+p^{l-j})^{-1} \end{array} \right)
\end{equation*}
for some $u, u' \in \Z_p^\times$, from which the following properties of these three groups follow:

\begin{eqnarray}
T(l,j) \times T(l,j)' \times T(l,j)'' & = & T(l-1, j-1), \\
\label{intersection}
T(l,j) \cap ( T(l,j)' \times T(l,j)'' ) & = & T(l,j-1),
\end{eqnarray}
and likewise for the intersection of $T(l,j)'$ or $T(l,j)''$ with the group generated by the other two.  We wish to show that $\dim M_{T(l,j)}$ decays exponentially as $j$ increases from $0$ to $l-1$.  If we suppose instead that $M_{T(l,j-1)} = M_{T(l,j)}$ for some $j \ge 2$, so that everything which is invariant under $T(l,j-1)$ is also invariant under $T(l,j)$, then by conjugation these vectors must also be invariant under $T(l,j)'$ and $T(l,j)''$, and hence by their product, which is $T(l-1, j-1)$.  This gives $\dim M_{T(l,j)} \le \dim M_{T(l-1,j-1)}$, and we know the latter dimension to be small by our inductive assumption.\\

To apply inclusion-exclusion, both spaces $M_{T(l,j)}$ and $M_{T(l,j)'}$ are quotients of $M_{T(l,j) \cap T(l,j)'}$, and their largest common quotient is equal to $M_{T(l,j) \times T(l,j)'}$.  We therefore have

\begin{eqnarray*}
\dim M_{T(l,j)} + \dim M_{T(l,j)'} & \le & \dim M_{T(l,j) \cap T(l,j)'} + \dim M_{T(l,j) \times T(l,j)'} \\
2 \dim M_{T(l,j)} & \le & \dim M_{T(l,j-1)} + \dim M_{T(l,j) \times T(l,j)'},
\end{eqnarray*}
by conjugation invariance and the intersection property (\ref{intersection}).  Applying this argument again to the pair of groups $T(l,j) \times T(l,j)'$ and $T(l,j)''$, we have

\begin{eqnarray*}
\dim M_{T(l,j)''} + \dim M_{T(l,j) \times T(l,j)'} & \le & \dim M_{ T(l,j)'' \cap  (T(l,j) \times T(l,j)') } + \dim M_{ T(l,j)'' \times T(l,j) \times T(l,j)' } \\
\dim M_{T(l,j)} + \dim M_{T(l,j) \times T(l,j)'} & \le & \dim M_{ T(l,j-1) } + \dim M_{ T(l-1,j-1)}.
\end{eqnarray*}
Combining the two inequalities gives

\begin{equation}
\label{induct}
3 \dim M_{T(l,j)} \le 2\dim M_{T(l,j-1)} + \dim M_{T(l-1,j-1)}.
\end{equation}

Now, let $l_0 \le k$ and $j_0$ be given, and suppose that our inductive hypothesis (\ref{indhyp}) holds for all $(l,j)$ with either $l < l_0$ or $l = l_0$ and $j < j_0$.  The inductive hypothesis follows trivially from the growth assumption (\ref{growthhyp}) when $j_0 = 0$ or $1$, so we may assume that $j_0 \ge 2$ and hence we are able to apply our inductive argument.  Substituting (\ref{indhyp}) into (\ref{induct}), we have

\begin{eqnarray*}
3 \dim M_{T(l_0,j_0)} & \le & 2 C \eta^{j_0-2} p^{2 l_0} + C \eta^{j_0-2} p^{2 l_0 - 2} \\
\dim M_{T(l_0,j_0)} & \le & C \eta^{j_0-2} p^{2 l_0} \left( \frac{2p^2 + 1}{ 3p^2} \right) \\
& \le & C \eta^{j_0-1} p^{2 l_0}.
\end{eqnarray*}
This establishes (\ref{indhyp}) for all $(l,j)$, from which the proposition follows by setting $l = k$ and $j = k-1$.

\end{proof}

\subsection{Proof of proposition \ref{coinvarlem}: the inductive step}

We now turn to the inductive step, and begin by describing why additional work is necessary in passing from $t = 1$ to $t > 1$.  When $t=1$, lemma \ref{single} may be thought of as a way of deforming the subgroups $G(p^k)$ to $T(p^k)$ while preserving some of the power saving of theorem \ref{harris}.  However, as was mentioned earlier, when $t>1$ and ${\bf k}$ is parallel such an approach is not possible because the initial bound for $M_{\GG_{\bf k}}$ provided by theorem \ref{harris} is larger than the trivial bound for $M_{\TT_{\bf k}}$.\\

How should we get around this?  Assume that $t > 1$, and factorise the groups and rings under consideration as

\begin{equation*}
\GG = G_1 \times \GG', \quad \TT_k = T_1(p^{k_1}) \times \TT'_k, \quad \text{and} \quad \Lambda = \Lambda_1 \otimes \Lambda'.
\end{equation*}
A natural approach to proosition \ref{coinvarlem} would be to consider $M_{\TT'_{\bf k}}$ as a $\Lambda_1$ module, and apply proposition \ref{single} to derive a bound for $M_{\TT_{\bf k}}$ from a bound for the coinvariant spaces $( M_{\TT'_{\bf k}} )_{G_1(p^l)}$ for $l \le k_1$.  Indeed, it would suffice to know that

\begin{equation}
\label{g1}
\dim M_{G_1(p^l) \times \TT'_{\bf k}} \ll p^{2l} | \GG' : \TT_{\bf k}' |
\end{equation}
uniformly in ${\bf k}$ and $l$.  However, $M_{\TT'_{\bf k}}$ need not be torsion for $\Lambda_1$ and so (\ref{g1}) will not hold in general.  We get around this difficulty by showing that the $\Lambda_1$-module $M_{\TT'_{\bf k}}$ may be written as an extension of $K$ by $L$, where $K$ satisfies (\ref{g1}) and $L$ has few generators.  $K_{\TT_{\bf k}}$ may be thought of as the part of $M_{\TT_{\bf k}}$ which may be bounded using the action of $\Lambda_1$, and $L_{\TT_{\bf k}}$ as the part which may be bounded using the action of $\Lambda'$.\\

We construct $L$ by using the action of $\Lambda_1$ to define a filtration of $M$ by $\Lambda'$ modules, one of which must be torsion.  To demonstrate this in a simple case, assume that $M_{G_1}$ is torsion as a $\Lambda'$ module.  If we also assume the inductive hypothesis that proposition \ref{coinvarlem} holds for the lower dimensional group $\GG'$, we obtain

\begin{equation*}
\dim ( M_{\TT'_{\bf k}} )_{G_1} \ll \eta^\kappa | \GG' : \TT_{\bf k}' |.
\end{equation*}
We may lift a basis for $( M_{\TT'_{\bf k}} )_{G_1}$ to a generating set for $M_{\TT'_{\bf k}}$ as a $\Lambda_1$ module, and (\ref{coinvarbd}) then follows trivially from this bound on the number of generators.  In this case, we therefore see that we may take $K = 0$, $L = M_{\TT'_{\bf k}}$.

In the general case, first assume without loss of generality that $M$ is a cyclic $\Lambda$ module.  To define the filtration we shall use, we need the following structure theorem for $\Lambda$, taken from \cite{DSMS}.

\begin{theorem}
\label{gpal}
Let $ g_1, \ldots, g_d $ be a topological generating set for $\GG$.  The completed group ring $\Lambda = \F_p \llbracket \GG \rrbracket$ is generated by $z_i = 1-g_i$, and every element of it can be uniquely expressed as a sum over multi-indices $\alpha$,

\begin{equation*}
x = \sum_\alpha \lambda_\alpha z^\alpha,
\end{equation*}
where $z^{\alpha} = \Pi_{i=1}^d z_i^{\alpha_i}$ and all $\lambda_\alpha \in \F_p$.  Moreover, all such sums are in $\Lambda$.  The filtration by degree gives $\Lambda$ the structure of a filtered ring whose associated graded ring is commutative, i.e. $z^\alpha z^\beta = z^{\alpha + \beta}$ up to terms of degree $> |\alpha| + |\beta|$.

\end{theorem}

When applied to $\Lambda_1$, theorem \ref{gpal} says that $\Lambda_1$ is an almost commutative polynomial algebra in three variables, equal to the $\F_p$ span of $z^\alpha = z_1^{\alpha_1} z_2^{\alpha_2} z_3^{\alpha_3}$ for $\alpha \in ( \Z_{\ge 0} )^3$.  Let us define an ordering $\succ$ on such triples $\alpha$ by requiring that $\alpha \succ \beta$ if $|\alpha| > |\beta|$, and ordering those triples $\alpha$ with a common value of $|\alpha|$ lexicographically.  We denote the successor of $\alpha$ under this ordering by $\alpha'$.  We shall write $\alpha \ge \beta$ if this inequality holds for each entry of the two indices.

Define $I_\alpha$ to be the subspace of $\Lambda_1$ spanned by all monomials $z^\beta$ with $\beta \succeq \alpha$, which is a two sided ideal by the almost commutativity of $\Lambda_1$, and if $N$ is any $\Lambda_1$ module we define $N_\alpha = I_\alpha N$.  If $\alpha \ge \beta$, multiplication by $z^{\alpha - \beta}$ induces an isomorphism $I_\beta / I_{\beta'} \simeq I_\alpha / I_{\alpha'}$, and hence a surjection

\begin{equation*}
z^{\alpha - \beta} : N_\beta / N_{\beta'} \longrightarrow N_\alpha / N_{\alpha'}.
\end{equation*}
We also note that the first quotient of the filtration $\{ N_\alpha \}$ of $N$ is equal to $N_{G_1}$.\\

We shall consider the filtration $\{ M_\alpha \}$ of $M$.  If $v \in M$ is a generator, the sucessive quotients $M_\alpha / M_{\alpha'}$ of this filtration are cyclic $\Lambda'$ modules generated by $z^\alpha v$.  Moreover, if we write $M$ as $\Lambda / I$ for some left ideal $I \subset \Lambda$ and let $\alpha$ be the last tuple in the ordering $\succ$ such that $I \subset I_\alpha \otimes \Lambda'$, we see that $M_\alpha / M_{\alpha'}$ is torsion for $\Lambda'$.  We then deduce that

\begin{equation}
\label{partbd}
\dim ( M_\alpha / M_{\alpha'})_{\TT'_k} \ll \eta^\kappa | \GG' : \TT_{\bf k}' |
\end{equation}
by applying the inductive hypothesis to $M_\alpha / M_{\alpha'}$ as a $\Lambda'$ module, and this will allow us to control the size of various filtered pieces of the $\Lambda_1$ module $N = M_{\TT'_k}$.  Because the projection $M \twoheadrightarrow N$ commutes with the action of $\Lambda_1$ it induces projections $M_\alpha \twoheadrightarrow N_\alpha$, and hence a surjection

\begin{equation*}
(M_\alpha / M_{\alpha'})_{\TT'_k} \twoheadrightarrow N_\alpha / N_{\alpha'}.
\end{equation*}
It then follows from (\ref{partbd}) that

\begin{equation}
\label{partbd2}
\dim N_\alpha / N_{\alpha'} \ll \eta^\kappa | \GG' : \TT_{\bf k}' |.
\end{equation}\\

Choose a subspace $\overline{L} \subset N_{G_1}$ for which the restriction

\begin{equation*}
z^\alpha : \overline{L} \longrightarrow N_\alpha / N_{\alpha'}
\end{equation*}
is an isomorphism, so that by (\ref{partbd2}) we have

\begin{equation*}
\dim \overline{L} \ll \eta^\kappa | \GG' : \TT_{\bf k}' |.
\end{equation*}
Choose a basis for $\overline{L}$, lift its elements to $N$, and let $L \subset N$ be the submodule they generate under the action of $\Lambda_1$.  Because $L$ has $\ll \eta^\kappa | \GG' : \TT_{\bf k}' |$ generators, we have

\begin{eqnarray*}
L_{T_1(p^k)} & \ll & \eta^\kappa |G_1 : T_1(p^k) | \times | \GG' : \TT'_{\bf k} | \\
& = & \eta^\kappa | \GG : \TT_{\bf k} |
\end{eqnarray*}
so that it suffices to bound the $T_1(p^k)$ coinvariants in $K = N/L$.  We have

\begin{equation*}
z^\alpha : L_{G_1} \twoheadrightarrow N_\alpha / N_{\alpha'}
\end{equation*}
by construction, so that $L_\alpha / L_{\alpha'} \twoheadrightarrow N_\alpha / N_{\alpha'}$.  We therefore have $L_\alpha + N_{\alpha'} = N_\alpha$, so $N_{\alpha'} + L = N_\alpha + L$.  However

\begin{equation*}
K_\alpha = (N_\alpha + L) / L \quad  \text{and} \quad K_{\alpha'} = (N_{\alpha'} + L) / L,
\end{equation*}
so that $K_\alpha / K_{\alpha'} = 0$ and hence $K_\beta / K_{\beta'} = 0$ for all $\beta \ge \alpha$.\\

Our last step is to use the triviality of $K_\beta / K_{\beta'}$ to bound $K_{G_1(p^l)}$, so that lemma \ref{single} may be applied.  Define the two sided ideals $I(p^l)$ and $I(p^l)_\beta$ of $\Lambda_1$ by

\begin{eqnarray*}
& 0 \longrightarrow I(p^l) \longrightarrow \F_p \llbracket G_1 \rrbracket \longrightarrow \F_p [ G_1 / G_1(p^l) ] \longrightarrow 0, & \\
& I(p^l)_\beta = I(p^l) + I_\beta. &
\end{eqnarray*}
$I(p^l)$ may also be described as the span of the monomials $z^\alpha$ for $\alpha \not\leq (p^l-1, p^l-1, p^l-1)$.  Define

\begin{equation*}
K(p^l) = I(p^l) K \quad \text{and} \quad K(p^l)_\beta = I(p^l)_\beta K,
\end{equation*}
so that $K_{G_1(p^l)} = K / K(p^l)$.  We have a map

\begin{equation}
\label{surject}
K_{G_1} \twoheadrightarrow K_\beta / K_{\beta'} \twoheadrightarrow K(p^l)_\beta / K(p^l)_{\beta'}
\end{equation}
for all $\beta$, so that

\begin{eqnarray}
\notag
\dim K(p^l)_\beta / K(p^l)_{\beta'} & \leq & \dim K_{G_1} \\
\notag
& \leq & \dim N_{G_1} \\
\label{remainder}
& \leq & | \GG' : \TT_{\bf k}' |.
\end{eqnarray}
The map (\ref{surject}) also implies that $K(p^l)_\beta / K(p^l)_{\beta'} = 0$ for all $\beta \ge \alpha$.  If $S_l$ is the set of indices $\beta \leq (p^l-1, p^l-1, p^l-1)$, the number of $\beta \in S_l$ with $\beta \not\leq \alpha$ is $\ll_\alpha p^{2l}$.  We may therefore apply the bound (\ref{remainder}) to obtain

\begin{eqnarray*}
\dim K_{G_1(p^l)} & \le & \sum_{\beta \in S_l} \dim K(p^l)_\beta / K(p^l)_{\beta'} \\
& \ll_\alpha & p^{2l}| \GG' : \TT_{\bf k}' |.
\end{eqnarray*}
This allows us to apply proposition \ref{single} to obtain the required bound for $\dim K_{T_1(p^k)}$, which completes the proof of proposition \ref{coinvarlem}.

\subsection{Extension to higher homological degree}

We finish this section by deducing the remaining statements of proposition \ref{induced} from proposition \ref{coinvarlem}, following Harris \cite{H} and Calegari and Emerton \cite{CE1}.  Recall that we must prove that if $M$ is a finitely generated $\Lambda$ module of rank $r$, then

\begin{eqnarray}
\label{hh0}
\dim M_{\TT_k} & = & (r + O(\eta^\kappa)) | \GG : \TT_{\bf k} |, \\
\label{hhi}
\dim H_i( \TT_k, M) & \ll & \eta^\kappa | \GG : \TT_{\bf k} |, \quad i \ge 1.
\end{eqnarray}

We begin with equation (\ref{hh0}), which we have just proven in the case $r = 0$.  In general, there is an exact sequence

\begin{equation*}
0 \longrightarrow T \longrightarrow M \longrightarrow M' \longrightarrow 0
\end{equation*}
with $T$ torsion and $M'$ torsion free of rank $r$, and by applying the rank 0 result to $T$ we deduce

\begin{equation*}
| \dim M_{\TT_k} - \dim M'_{\TT_k} | \ll \eta^\kappa | \GG : \TT_{\bf k} |,
\end{equation*}
so that we may assume $M$ is torsion free.  This implies the existence of morphisms $\Lambda^r \rightarrow M$ and $M \rightarrow \Lambda^r$ with torsion cokernels, from which (\ref{hh0}) follows from the associated long exact sequences on homology and $\dim \Lambda^r_{\TT_k} = r | \GG : \TT_{\bf k} |$.\\

Turning to $i > 0$, because $M$ is finitely generated there exists a short exact sequence

\begin{equation*}
0 \longrightarrow N \longrightarrow \Lambda^n \longrightarrow M \longrightarrow 0
\end{equation*}
of finitely generated $\Lambda$ modules for some $n \ge 0$.  Because $\Lambda^n$ is acyclic, the associated long exact sequence in homology gives

\begin{eqnarray}
\label{longex1}
0 \longrightarrow H_1(\TT_k, M) \longrightarrow N_{\TT_k} \longrightarrow \Lambda^n_{\TT_k} \longrightarrow M_{\TT_k} \longrightarrow 0, \\
\label{longex2}
H_i(\TT_k, M) \simeq H_{i-1}(\TT_k, N), \quad i \ge 2.
\end{eqnarray}
The lemma for $i=1$ now follows from (\ref{longex1}) and (\ref{hh0}), taking into account the fact that rank is additive in short exact sequences.  For higher $i$, we use induction.  Assume the result for all $i \le m$ and all finitely generated modules, in particular for $N$.  We then obtain it for $M$ in degree $m+1$ from the isomorphism (\ref{longex2}), which completes the proof.

\section{The structure of $\F_p[G / H(p^k)]$}
\label{cosetrep}

The goal of this section is to prove the following structure result, which will allow us to decompose a subrepresentation $L \subset \F_p[G / H(p^k)]$ into pieces to which Shapiro's lemma can be applied.  More precisely, we shall filter $L$ with quotients isomorphic to $\F_p[G / H(p^l)]$ for $l \le k$, in a manner analogous to the base $p$ expansion of an integer.

\begin{prop}
\label{quotients}
$\F_p[G / H(p^k)]$ has a filtration $0 = F(0) \subset \ldots \subset F(p^{k-1}) = \F_p[G / H(p^k)]$ such that for all $1 \le l \le k$ and all $0 \le a < p^{k-l}$, $F((a+1)p^{l-1}) / F(ap^{l-1}) \simeq \F_p[G / H(p^l)]$.  Moreover, $F(i)$ is the unique subrepresentation of $\F_p[G / H(p^k)]$ of dimension $i$.
\end{prop}

The filtration of $L$ which we may construct from this is described below.

\begin{cor}

Let $L \subset \F_p[G / H(p^k)]$ be a submodule of dimension $d$, and let the base $p^4$ expansion of $d$ be written

\begin{equation*}
d = \sum_{i=1}^l p^{4\alpha(i)},
\end{equation*}
where $\alpha(i)$ is a non-increasing sequence of non-negative integers.  Then there exists a filtration $0 = L_0 \subset \ldots \subset L_l = L$ of $L$ by submodules $L_i$ such that $L_i / L_{i-1} \simeq \F_p[ G / H( p^{4\alpha(i)+1} )]$.

\end{cor}

\begin{proof}

Let the partial sums of the expansion of $d$ be

\begin{equation*}
s(i) = \sum_{j=1}^i p^{4\alpha(j)},
\end{equation*}
and let $L_i = F(s(i))$.  We have $s(i) = s(i-1) + p^{4\alpha(i)}$ and $p^{4\alpha(i)} | s(i-1)$, so by the proposition $L_i / L_{i-1} \simeq \F_p[ G / H( p^{4\alpha(i)+1} )]$ as required.

\end{proof}

We begin the proof of proposition \ref{quotients} by defining $\phi: G \rightarrow p \Z_p$ by

\begin{equation*}
\phi: \left( \begin{array}{cc} a & b \\ c & d \end{array} \right) \mapsto \frac{c}{a}.
\end{equation*}

It can be seen that $\phi$ intertwines the right action of $G$ on itself with the action on $p\Z_p$ by fractional linear transformations given by

\begin{equation*}
g : z \mapsto \frac{dz + c}{bz + a}.
\end{equation*}
Moreover, if we let $\overline{\phi}$ denote the composition of $\phi$ with reduction modulo $p^k \Z_p$, then it may be seen that $\overline{\phi}$ factors through the right action of $H(p^k)$.  $\phi$ therefore defines an isomorphism between the actions of $G$ on $G/H(p^k)$ and $p\Z_p / p^k$, and hence the representations $\F_p[G/H(p^k)]$ and $\F_p[ p\Z_p / p^k]$.  There is an important colllection of subspaces of $\F_p[ p\Z_p / p^k]$, which we shall denote by $F(i)$ for $0 \le i \le p^{k-1}$, and which may be defined as the space of all functions on $p\Z_p$ obtained by taking the reductions modulo $p$ of polynomials $p : \Q_p \rightarrow \Q_p$ of degree at most $i-1$ which take integer values on that set (and where we set $F(0) = 0$).  It is a theorem of Lucas that such functions are in fact constant on cosets of $p^k \Z_p$, and moreover that a basis for $F(i)$ is given by the binomial coefficients $x \mapsto \binom{x/p}{t}$, $0 \le t \le i-1$, so that $\dim F(i) = i$.  The lower unipotent matrix

\begin{equation*}
\overline{N} = \left( \begin{array}{cc} 1 & 0 \\ -p & 1 \end{array} \right)
\end{equation*}
acts on $p\Z_p$ by $z \mapsto z-p$, and so acts on the function $\binom{x/p}{t}$ by

\begin{eqnarray*}
\overline{N} \binom{x/p}{t} & = & \binom{x/p+1}{t} \\
& = & \binom{x/p}{t} + \binom{x/p}{t-1}.
\end{eqnarray*}
It follows from this that $F(i)$ is the unique $i$ dimensional subspace of $\F_p[ p\Z_p / p^k]$ which is invariant under $\overline{N}$.  In fact, more is true:

\begin{lemma}
\label{symlattice}
The subspaces $F(i)$ are invariant under $G$.
\end{lemma}

\begin{proof}
For $0 \le d \le p^{k-1}-1$ let $\Sym^d$ be the standard $d$th symmetric power representation of $G$, realised on the space of functions $f : G \rightarrow \Q_p$ of the form

\begin{equation*}
f: \left( \begin{array}{cc} a & b \\ c & d \end{array} \right) \mapsto p(a,c)
\end{equation*}
for some homogeneous polynomial $p$ of degree $d$, and let $\VV_d \subset \Sym^d$ be the lattice of integral valued functions.  For $f \in \VV_d$, we have

\begin{eqnarray*}
f \left[ \left( \begin{array}{cc} a & b \\ c & d \end{array} \right) \right] & = & p(a,c)\\
& = & a^d p(1, \tfrac{c}{a})\\
& \equiv & p(1, \tfrac{c}{a}) \quad \text{mod} \: p,
\end{eqnarray*}
so that when we transfer functions in $\VV_d$ to $p\Z_p$ via $\phi$ and reduce modulo $p$ we obtain exactly the space $F(d)$.  $\VV_d$ is clearly preserved by $G$, and because $\phi$ was an intertwiner we see that $F(d)$ is also.

\end{proof}

In light of this lemma, we see that $F(i)$ are exactly the subspaces described in the filtration of proposition \ref{quotients}.  The key idea behind the main claim of proposition \ref{quotients} is a certain recursive characterisation of the subspaces $F(i)$.  If we define $F(a,l)$ to be the subspace of $\F_p[p^l \Z_p / p^k]$ obtained by reducing polynomials of degree at most $a-1$ which are integral valued on $p^l \Z_p$ modulo $p$, we then see that $F(a p^{l-1})$ is exactly the subspace of $\F_p[ p\Z_p / p^k]$ consisting of elements whose restrictions to the cosets $\F_p[ z + p^l \Z_p / p^k ]$ all lie in $F(a, l)$.  Indeed, this follows from the observation that both spaces have the same dimension, and are invariant under the lower unipotent subgroup.  We therefore have

\begin{eqnarray}
\label{quotiso}
F((a+1)p^{l-1}) / F(ap^{l-1}) & \simeq & ( F(a+1, l) / F(a, l) ) \otimes \F_p[ p\Z_p / p^l ] \\
\notag
& \simeq & \F_p[ p\Z_p / p^l ]
\end{eqnarray}
as vector spaces, and so the proposition would follow from knowing that the identification (\ref{quotiso}) commutes with the action of $G$ on both sides.  To show this, let $g \in G$ be given.  By restriction, $g$ gives a map from $z + p^l \Z_p / p^k$ to $gz + p^l \Z_p / p^k$, and when we choose an identification of both of these sets with $p \Z_p / p^{k-l+1}$ in the natural way we see that this map is equal to a fractional linear transformation $x \mapsto g' x$ for some $g' \in G$.  Therefore, using our claim that the fractional linear action of $G$ preserves $F(a, l)$, we see that the map

\begin{equation*}
\F_p[z + p^l \Z_p / p^k] \longrightarrow \F_p[gz + p^l \Z_p / p^k]
\end{equation*}
preserves $F(a, l)$ and $F(a+1,l)$, and because $F(a+1,l)/F(a,l)$ is one dimensional it acts trivially on the quotient.  It follows that the isomorphism (\ref{quotiso}) commutes with $G$, which concludes the proof.

\section{Completed homology}
\label{compcoh}

In this section, we shall construct the $p$-adic local system $V_{\bf d}$ with the property that

\begin{equation}
\label{coeffhom}
\dim_\C H_i( Y, W_{\bf d} ) = \dim_{\Q_p} H_i( Y, V_{\bf d} ).
\end{equation}
We shall then apply the theory of $p$-adically completed cohomology developed by Calegari and Emerton \cite{CE1, CE2, E} to convert the problem of bounding the RHS to one of bounding $H_i( \GG, M \otimes V_{\bf d})$, where $M$ is a fixed $\Lambda_{\Q_p}$ modules and ${\bf d}$ varies.  To begin, let $p$ be a prime which is totally split in $F$.  If $\{ \p_1, \ldots, \p_n \}$ are the primes of $F$ above $p$, and $F_{\p_i}$ is the completion of $F$ at $\p_i$, $\Gamma$ has an embedding

\begin{equation*}
\phi: \Gamma \longrightarrow SL_2(F_p) := \prod_i SL_2(F_{\p_i})
\end{equation*}
the closure of whose image is a compact open subgroup of the target $p$-adic group.  If we let $\GG$ be as in (\ref{Gdef}) with $t = n$, $SL_2(F_p)$ contains $\GG$ as a compact open subgroup, and by passing to a finite index sublattice we may assume that $\phi(\Gamma) \subset \GG$.  We shall in fact assume that $\overline{ \phi(\Gamma) } = \GG$, which is not necessary for the proof but allows us to avoid cluttering the previous sections with excessive notation.  In any case, it may always be arranged after first choosing $p$ at which $\Gamma$ has full level.\\

Recalling our convention that ${\bf d}$ was an $n$-tuple of non-negative integers, we define the representation $V_{\bf d}$ of $\GG$ to be the tensor products of the representations $\Sym^{d_i}$ of $G_i$, and denote both the restriction of $V_{\bf d}$ to $\Gamma$ under $\phi$ and the associated local system on $Y$ in the same way.  Note that this definition relies on a choice of bijection between the complex embeddings $\{ \sigma_i \}$ and $p$-adic embeddings $\{ \p_i\}$ of $F$.  The following lemma shows that (\ref{coeffhom}) holds if this choice is made in a natural way.

\begin{lemma}
\label{coeff}
Let $\overline{F}$ be the Galois closure of $F$.  There exist a complex and $p$-adic place $\sigma$ and $\p$ of $\overline{F}$, and a bijection between the set of all complex and $p$-adic places $\{ \sigma_i \}$ and $\{ \p_i\}$ of $F$, such that for all ${\bf d} \in (\Z_{\ge 0})^n$ there exists a representation $\rho$ of $\Gamma$ over $\overline{F}$ such that

\begin{itemize}

\item $\rho \otimes_{\sigma} \C \simeq W_{\bf d}$,

\item $\rho \otimes_\p \Q_p \simeq V_{\bf d}$.

\end{itemize}

\end{lemma}

\begin{proof}

Let $G = \text{Gal}(\overline{F}/\Q)$, and $H = \text{Stab}_G(F)$.  If $\sigma = \sigma_0$ is a chosen complex embedding of $\overline{F}$ and $g_i \in G / H$ is a fixed system of coset representatives, the set of all complex embeddings of $F$ is equal to the restrictions of $\sigma_i = \sigma \circ g_i$.  Likewise, if we choose a $p$-adic place $\p$ of $\overline{F}$, the restrictions of $\p_i = \p \circ g_i$ form a complete set of $p$-adic embeddings of $F$.

Let $\Sym^d$ denote the $d$th symmetric power representation of $SL_2(\overline{F})$.  If we define the representation $\rho$ of $SL_2(\overline{F})$ by

\begin{equation*}
\rho \simeq \bigotimes_i \Sym^{d_i} \circ g_i,
\end{equation*}
then on restriction to $SL_2(F)$ we have

\begin{eqnarray*}
\rho \otimes_{\sigma} \C & \simeq & \bigotimes_i ( \Sym^{d_i} \circ g_i ) \otimes_{\sigma} \C \\
& \simeq & \bigotimes_i \Sym^{d_i} \otimes_{\sigma_i} \C \\
& \simeq & W_{\bf d},
\end{eqnarray*}
and likewise for $V_{\bf d}$.  The lemma follows on restriction to $\Gamma$.

\end{proof}

We now introduce the $p$-adic tools which we shall use to study the RHS of (\ref{coeffhom}).  Let $\PP \subset \{ \p_1, \ldots, \p_n \}$ be the set of places at which we are allowing the weight to vary, so that $\PP$ is the complement of $I$ under the bijection of lemma \ref{coeff}.  Let

\begin{equation*}
\GG_\PP = \prod_{ \p_i \in \PP} G_i,
\end{equation*}
and factorise $V_{\bf d}$ as $V_{ {\bf d}, \PP} \otimes V_{\bf d}^\PP$.  We shall choose a $\GG$-stable lattice $\VV_{\bf d} \subset V_{\bf d}$ using lemma \ref{symlattice}, by letting ${\bf k}$ be the smallest $t$-tuple of integers $\ge 1$ satisfying $p^{k_i-1} > d_i$ and $4 | k_i - 1$, and choosing $\VV_{d_i} \subset V_{d_i}$ such that $\VV_{d_i} / p \subset \F_p[G_i / H_i(p^{k_i})]$ for all $i$.  We let $\VV_{\bf d} = \otimes \VV_{d_i}$, which we factorise as $\VV_{ {\bf d},\PP} \otimes V_{\bf d}^\PP$.  Let

\begin{equation*}
\GG_{\PP,r} = \prod_{ \p_i \in \PP} G_i(p^r)
\end{equation*}
be the principal congruence subgroups of $\GG_\PP$, $\Gamma_r = \Gamma \cap \GG_{\PP,r}$, and $Y_r$ be the corresponding covers of $Y$.  Following Emerton, we define

\begin{equation*}
\widetilde{H}_i(\VV_{\bf d}) = \underset{ \substack{\longleftarrow \\ s} }{\lim} \, \underset{ \substack{\longleftarrow \\ r} }{\lim} \; H_i( Y_r, \VV_{\bf d} / p^s )
\end{equation*}
to be the $i$th completed homology module of the tower $\{ Y_r \}$ with coefficients in $\VV_{\bf d}$.  We shall use the following fact about these modules, taken from \cite{CE1, E}.

\begin{enumerate}

\item $\widetilde{H}_i(\VV_{\bf d})$ is a $p$-adically complete and separated $\Z_p$ module.

\item $\widetilde{H}_i(\VV_{\bf d})$ has the structure of a finitely generated $\Z_p \llbracket \GG_\PP \rrbracket$ module which extends the natural action of $\GG_\PP$ by conjugation.

\item Because $SL(2,\C)$ does not admit discrete series, $\widetilde{H}_i(\VV_{\bf d})$ is a torsion $\Z_p \llbracket \GG_\PP \rrbracket$ module for all $i$.

\item $\widetilde{H}_i(\VV_{\bf d})$ carries a natural action of $SL(2,F_{\p_i})$ for those $\p_i \in \PP$ satisfying $d_i = 0$, which extends the action of $G_i$.

\item There is a spectral sequence

\begin{equation}
\label{specseq}
E^{i,j}_2 = H_i ( \GG_\PP, \widetilde{H}_j(\VV_{\bf d}) \otimes_{\Z_p} \Q_p ) \Longrightarrow H_{i+j}( Y, V_{\bf d}).
\end{equation}

\end{enumerate}

The spectral sequence (\ref{specseq}) implies an upper bound

\begin{equation}
\label{specbound}
\dim H_q( Y, V_{\bf d}) \le \sum_{i+j = q} \dim H_i ( \GG_\PP, \widetilde{H}_j(\VV_{\bf d}) \otimes_{\Z_p} \Q_p )
\end{equation}
for the classical homology group we are interested in, and after some simplifications this will reduce the problem of a power saving to a statement about torsion $\Z_p \llbracket \GG_\PP \rrbracket$ modules with a compatible $SL_2$ action which will follow from the results of sections \ref{coinvar} and \ref{cosetrep}.  Because we have defined the representation $\VV_{\bf d}$ by pulling back a representation of $\GG$, $\VV_{{\bf d},\PP} / p^s$ is eventually trivial as a representation of $\Gamma_r$.  We therefore have

\begin{eqnarray*}
\widetilde{H}_j(\VV_{\bf d}) & \simeq & \underset{ \substack{\longleftarrow \\ s} }{\lim} \, \underset{ \substack{\longleftarrow \\ r} }{\lim} \; H_i( Y_r, \VV_{\bf d}^\PP / p^s ) \otimes \VV_{{\bf d},\PP} / p^s \\
& \simeq & \widetilde{H}_j(\VV_{\bf d}^\PP) \otimes \VV_{{\bf d},\PP}
\end{eqnarray*}
as representations of $\GG_\PP$, and because we are fixing $d_i$ for those $\p_i \not\in \PP$ we shall simply write $\widetilde{H}_j$ for $\widetilde{H}_j(\VV_{\bf d}^\PP)$ and $\widetilde{H}_{j, \Q_p}$ for $\widetilde{H}_j(\VV_{\bf d}^\PP) \otimes_{\Z_p} \Q_p$.  Furthermore, as we do not need to give any further consideration to the primes not in $\PP$ we shall ignore them from this point on and write $\GG$ for $\GG_\PP$, $V_{\bf d}$ for $\VV_{{\bf d},\PP}$, and assume that $\PP = \{ \p_1, \ldots, \p_t \}$.  The upper bound (\ref{specbound}) may then be rewritten

\begin{equation}
\label{specbound2}
\dim H_q( Y, V_{\bf d}) \le \sum_{i+j = q} \dim H_i ( \GG, \widetilde{H}_{j, \Q_p} \otimes V_{\bf d} ).
\end{equation}

\section{Reduction modulo $p$}
\label{modp}

We now combine the results of the previous sections to prove an upper bound for the RHS of (\ref{specbound2}), by choosing a lattice inside $\widetilde{H}_{j,\Q_p} \otimes V_{\bf d}$ which we then reduce modulo $p$.  The lattice we take will be the tensor product of the image of $\widetilde{H}_j$ in $\widetilde{H}_{j,\Q_p}$, and the lattice $\VV_{\bf d} \subset V_{\bf d}$ constructed above using lemma \ref{symlattice}.  We know that $\VV_{\bf d} / p$ is a submodule of $\F_p[\GG / \Hh_{\bf k}]$ which we denote by $L$.

The image of $\widetilde{H}_j$ in $\widetilde{H}_{j,\Q_p}$ is isomorphic to the $p$-torsion free quotient $\widetilde{H}_{j, \text{tf}}$ of $\widetilde{H}_j$, and we denote the reduced lattice $\widetilde{H}_{j, \text{tf}} / p$ by $M_j$.  By property (2), $\widetilde{H}_{j,\Q_p}$ is a torsion $\Lambda_{\Q_p}$ module which implies that $M_j$ is a torsion $\Lambda$ module.  Moreover, $\widetilde{H}_{j, \text{tf}}$ is invariant under $SL(2,F_{\p_i})$ for any $\p_i$, so that $M_j$ also carries an action of these groups.  We may now reduce our chosen lattice modulo $p$ to obtain

\begin{equation*}
\dim_{\Q_p} H_i( \GG, \widetilde{H}_{j,\Q_p} \otimes V_{\bf d}) \le \dim_{\F_p} H_i( \GG, M_j \otimes L),
\end{equation*}
and the required bound on the RHS is obtained by combining the results of sections \ref{coinvar} and \ref{cosetrep} in the following lemma.

\begin{lemma}
\label{homsave1}
Let $M$ be a torsion $\Lambda$ module with a compatible action of $SL_2$, and $L$ any subrepresentation of $\F_p[\GG / \Hh_{\bf k}]$ which factorises as $\otimes L_i$ with $L_i \subset \F_p[ G_i / H_i(p^k) ]$.  We then have

\begin{equation}
\label{homsave}
\dim H_i( \GG, M \otimes L) \ll \alpha^\kappa | \GG : \Hh_{\bf k} |
\end{equation}
for all $i$, where $\alpha = \eta^{1/2}$ and the implied constant depends only on $M$.
\end{lemma}

\begin{proof}

In the case $L = \F_p[\GG / \Hh_{\bf k}]$, this follows easily from Shapiro's lemma and lemma \ref{induced} once we use the $SL_2$ action to conjugate $\Hh_{\bf k}$ to $\TT_{{\bf k}'}$, where $k_i' = (k_i + 1)/2$ (this is where our divisibility assumption is used).  For general $L$, we apply lemma \ref{quotients} to the factors $L_i$ to obtain a filtration $L = F_0 \supset F_1 \supset \ldots$ such that every quotient $F_i / F_{i+1}$ is isomorphic to $\F_p[\GG / \Hh_{\bf l}]$ for some ${\bf l} \le {\bf k}$ with $4 | l_i - 1$, and each isomorphism class of quotient occurs at most $p^{4t}$ times.  By applying (\ref{homsave}) in the case $L = \F_p[\GG / \Hh_{\bf l}]$ to the quotient modules, we obtain

\begin{eqnarray*}
\dim H_i( \GG, M \otimes L) & \le & p^{4t} \sum_{ \substack{ l \le k \\ l \equiv 1 \: (4) } } H_i( \Hh_{\bf l}, M) \\
& \ll & \sum_{l \le k} \alpha^{\min(l_i)} | \GG : \Hh_{\bf l} | \\
& \ll & \sum_{i = 1}^t \sum_{l \le k} \alpha^{l_i} | \GG : \Hh_{\bf l} | \\
& \ll & \sum_{i = 1}^t \alpha^{k_i} | \GG : \Hh_{\bf k} | \\
& \ll & \alpha^{\kappa} | \GG : \Hh_{\bf k} |.
\end{eqnarray*}

\end{proof}

It remains to follow the bound we have proven back to one for the original cohomology group $H_i(Y, W_{\bf d})$.  Lemma \ref{homsave1} gives

\begin{equation*}
\dim H_i( Y, V_{\bf d}) \ll \alpha^\kappa | \GG : \Hh_{\bf k} |,
\end{equation*}
where ${\bf k}$ is a $|\PP|$-tuple of integers satisfying $| k_i - \ln_p d_i | \le 4$ for $i \in \PP$.  We therefore have

\begin{eqnarray*}
\dim H_i( Y, V_{\bf d}) & \ll & \alpha^\kappa \prod p^{k_i} \\
& \ll & ( \underset{i \in \PP}{\min}\: d_i )^{\ln \alpha / \ln p} \prod_{i \in \PP} d_i.
\end{eqnarray*}
When expressed in terms of the original notation for $W_{\bf d}$ with ${\bf d} \in ( \Z_{\ge 0} )^{r_1 + r_2}$, we see that this is equivalent to proposition \ref{main2}, and hence theorem \ref{main}.

\end{document}